\documentclass[10pt]{amsart}
\usepackage{amssymb,amsmath,amsthm,amscd,mathrsfs,graphicx}
\usepackage[cmtip,all]{xy}
\usepackage[normalem]{ulem}
\usepackage{color}

\numberwithin{equation}{section}
\newtheorem{teo}{Theorem}[section]
\newtheorem{pro}[teo]{Proposition}
\newtheorem*{mainthm}{Main Theorem}

\newtheorem{cor}[teo]{Corollary}
\newtheorem{fact}[teo]{Fact}

\theoremstyle{definition}
\newtheorem{dfn}[teo]{Definition}
\newtheorem{exa}[teo]{Example}

\theoremstyle{remark}
\newtheorem{rem}[teo]{Remark}

\newcommand{\sX}{\mathscr{X}}

\newcommand{\bC}{\mathbb{C}}
\newcommand{\bP}{\mathbb{P}}
\newcommand{\sH}{\mathscr{H}}

 \newcommand{\Art}{\mathrm{Art}}
  \newcommand{\Spec}{\mathrm{Spec}}
\begin{document}
\bibliographystyle{amsalpha}

%%%%%%%%%%%%%%%%%%%%%%%%%%%%%%%%%%%%%%%%%
%%% Title page
\title[Semi-stable reduction for ADE curves]{Simultaneous semi-stable reduction for curves with ADE singularities}

\author[S. Casalaina-Martin]{Sebastian Casalaina-Martin }
\address{University of Colorado at Boulder, Department of Mathematics,   Boulder, CO 80309}
\email{casa@math.colorado.edu}
\thanks{The  first author was partially supported by NSF grant DMS-1101333.}

\author[R. Laza]{Radu  Laza}
\address{Stony Brook University, Department of Mathematics,  Stony Brook, NY 11794}
\email{rlaza@math.sunysb.edu}
\thanks{The second author was partially supported by NSF grant DMS-0968968 and a Sloan Fellowship}

%\date{\today}

\begin{abstract} 
A key tool in the study of algebraic surfaces and their moduli is Brieskorn's simultaneous resolution for families of algebraic surfaces with simple (du Val or ADE) singularities.  In this paper we show that a similar statement holds for families of curves with at worst  simple (ADE)  singularities.  For a family $\mathscr X\to B$ of  ADE  curves, we give an explicit and natural resolution of the rational map $B\dashrightarrow \overline M_g$.  Moreover,  we discuss a lifting of this map to the moduli stack  $ \overline {\mathcal M}_g$, i.e. a simultaneous semi-stable reduction for the family $\mathscr X/B$.  In particular, we note that  in contrast to what might be expected from the case of surfaces, the natural Weyl cover of $B$ is not a sufficient base change for a lifting of the map $B\dashrightarrow \overline M_g$ to $\overline {\mathcal M}_g$.
\end{abstract}
\maketitle

%%%%%%%%%%%%%%%%%%%%%%%%%%%%%%%%%%%%%%%%%
%%% S0 -- Introduction
\section*{Introduction}
As a consequence of the Deligne-Mumford  Stable Reduction Theorem \cite{dm}, many questions regarding curves and their moduli can be studied effectively by reducing to the case of  one parameter families of curves having at worst nodes as singularities.  Consequently, a great deal of research has been devoted to studying such families, and much is known. Recently there has been  a growing interest in understanding families of curves with singularities worse than nodes.  For instance, in a program started by B. Hassett and S. Keel (see \cite{hh}), moduli spaces of such curves have arisen naturally in the study of the canonical model of $\overline {M}_g$ (see also \cite{hh2}, \cite{hl} and \cite{smyth09}).  In another direction, such spaces have arisen  in the authors' study of compactifications of Kondo's ball quotient models of $M_3$ and $M_4$  (see \cite{k2} and \cite[\S8]{casalaza}).

With this as motivation,  it is natural to ask the following question: {\it Assuming that a moduli space $\overline {M}_g^T$ of curves with a given class of singularities were constructed, is it possible to understand in a systematic way the birational relationship between $\overline {M}_g^T$  and $\overline {M}_g$?} In this paper we answer this question in a situation that,  while restricted, often arises in practice.  Namely, we consider curves with simple (or ADE) singularities and focus on the local moduli space (mini-versal deformation space) instead of a global moduli space. The two restrictions are quite natural: ADE singularities, the most basic singularities beyond nodes, are also the first to be encountered in the Hassett-Keel program, and the local restriction avoids some subtle questions involved in the definition of a good moduli functor. 
 
The main result of this paper is an explicit construction of an  \'etale local resolution of the moduli map $B\dashrightarrow \overline{M}_g$ for a family of curves $\mathscr X\to B$ with  ADE  singularities. The process used to obtain the resolution is similar in spirit to Brieskorn's simultaneous resolution of  singularities for families of surfaces with  ADE  singularities (see \cite{briessing} and \cite{tjurina0}).

We state a concise version of our result below using the notion of the Weyl cover and wonderful blow-up. These  are explicit maps, described in detail in the body of the paper,  that are determined by the root systems associated to the singularities 
(Definitions \ref{dfnwc}, \ref{dfnclr}, and \ref{dfndc}).   A brief description is also given in the outline of the paper below.   Note that since we are considering the resolution question \'etale locally, it suffices to understand the case where $\mathscr X\to B$ is a mini-versal deformation of an  ADE   curve $X$.

\begin{mainthm}
Let $\mathscr X \to B_X$ be a mini-versal deformation of an  ADE   curve $X$ with $p_a(X)=g\ge 2$.  The wonderful blow-up of the Weyl cover of $B_X$ resolves the rational moduli map to the moduli scheme $\overline{M}_g$, but fails to resolve the rational moduli map to the moduli stack $\overline{\mathcal M}_g$ along the $A_{2n}$ locus of the discriminant ($n\in \mathbb N$).  
The addition of a stack structure (generically $\mathbb Z/2\mathbb Z$ stabilizers) along this locus resolves the moduli map to $\overline{\mathcal M}_g$.
\end{mainthm}

The theorem is proven, and stated more precisely, in Theorem \ref{teoresmm},  and  Corollaries \ref{teossr1}, \ref{teossr2}. In the text we also discuss the geometric meaning of the construction and its relationship to the Hassett-Keel program. In particular, we note that the results of Hassett \cite{hassettstable} are closely related to ours.  For instance, they can be used to describe the proper transform of the discriminant  $\Delta_X\subseteq B_X$ in $\overline M_g$ (see \S \ref{sechassettstable}).

We also point out that for a generically stable family of curves
$\mathscr X \to B$ inducing a rational map $f:B \dashrightarrow
\overline{\mathcal M}_g$, it is known there exists a generically finite
proper morphism $\tilde B \to B$ such  that the induced rational map $\tilde
B\dashrightarrow \overline{\mathcal M}_g$ extends to a morphism
(\cite[Theorem 5.8]{dejong1}, \cite[Theorem 2.7]{edidinetal}).  In
light of this, the content of our theorem is to give an explicit \'etale local description of such a  generically finite morphism $\tilde B\to B$ resolving the moduli map $B\dashrightarrow \overline {\mathcal M}_g$. This is important for the applications we have in mind (e.g. the study of the ball quotient model of $\overline{ M}_4$) and presumably for other questions related to the Hassett-Keel program.

\subsection*{Outline of the paper} We start in Section \ref{secdef}  by recalling some basic facts about the structure of the deformation space $B_X$ of a reduced Gorenstein curve $X$.  
This essentially allows us to study   the deformations of a complete curve
with ADE singularities using results on the deformations of the  singularities themselves.
Next, in Section \ref{sectcover}, using standard results from singularity theory, we get a good hold on both the deformation space and the discriminant of an ADE curve.  

 More precisely, the singularities of $X$ determine a Weyl group $W_X$: to each singularity of type  $A_n$, $D_m$, or $E_r$, one associates the corresponding Weyl group $W(A_n)$, $W(D_m)$, or $W(E_r)$ respectively, and  $W_X$ is the product of these groups.  Associated to this group is a finite \emph{Weyl cover} of the deformation space.  
 From a result of Brieskorn \cite{briessing} we obtain that the Weyl cover of the
deformation space is smooth,  and the pull-back of the discriminant is an arrangement of hyperplanes whose combinatorics are governed by the root system associated to the Weyl group.   

In \S \ref{secmssch}, we discuss the so-called {\it wonderful blow-up} introduced by de Concini-Procesi \cite{dcp} (see also \cite{macpherson}, \cite{hu}). Specifically, one blows-up an arrangement of hyperplanes inductively starting with the highest codimension (irreducible) linear strata. In our situation, by applying the Weyl cover followed by this blow-up, we modify the versal deformation family $\mathscr X\to B_X$ to obtain a family of curves over a base $\widetilde{B}_X$ with normal crossing discriminant. By applying general extension results of de Jong-Oort \cite{dejongoort} (see also \cite{cautis}, \cite{mochizuki}), we obtain that the \emph{a priori} rational map $\widetilde{B}_X\dashrightarrow  \overline{M}_g$ extends to a regular map to the coarse moduli space $\overline{M}_g$. We point out that in the analogous case of configuration space for points on the complex line, one obtains a similar result via the Fulton-MacPherson compactification (see \S\ref{confspace}).

For more subtle geometric questions it is important to lift the map $\tilde B_X\to \overline M_g$ to the stack $\overline{\mathcal M}_g$, i.e. to obtain a simultaneous stable reduction for families of curves with ADE singularities. We discuss this issue in \S \ref{secmstack}. In short, using the full strength of the de Jong-Oort theorem cited above, once we are in the normal-crossing case $(\widetilde B,\tilde \Delta)$, to get an extension to  $\overline{\mathcal M}_g$ it suffices to  establish the unipotency of certain monodromy representations at the generic points of the boundary divisors. Using results on Artin groups  (see \cite{brieskornartin})  and singularities, we establish the surprising fact that while the map $\widetilde{B} \dashrightarrow \overline{\mathcal M}_g$ extends along most of the boundary, there is an obstruction (due to the hyperelliptic involution of the tails) to extending the map along the loci in $\tilde \Delta$ corresponding to curves with $A_{2n}$ singularities.  

In \S \ref{secmstack} we discuss how this obstruction can be eliminated by adding an appropriate stack structure along the $A_{2n}$ boundary divisors.  To obtain a resolution in the category of schemes, one can consider introducing a level $n$ structure that will kill (for $n>2$) the monodromy obstruction (see \cite{llevel}, \cite[Theorem A(3)]{mochizuki}).
For clarity we discuss in Section \ref{secandirect} the explicit construction of the stable reduction in the $A_n$ case. This computation is essentially contained in \cite[\S3]{tjurina0}, \cite{bruber} 
(see also \cite[\S5]{maksym}).  Finally, some of the results of this paper can be extended to higher dimensional varieties and their period maps.  The authors are currently investigating applications to the period map for cubic threefolds.

%%%%%%%%%%%%%%%%%%%%%%%%%%%%%%%%%%%%%%%%%
%%% Acknowledgements
\subsection*{Acknowledgements} 
While preparing this manuscript, we  learned that B. Hassett and P. Hacking (related to their work \cite{hassettstable} and \cite{hacking} respectively) were essentially aware of the semi-stable reduction process described in this paper. We are thankful for several discussions they   had with us on the subject. We are also grateful to B. Hassett for sharing with us some preliminary notes,  and encouraging our pursuit of the subject. We thank D. Allcock, M. Fedorchuk,  R. Friedman, D. Smyth and F. van der Wyck for helpful discussions on various specific points of the paper.  Finally, we would like to thank the MSRI for its hospitality during the preparation of this manuscript.

%%%%%%%%%%%%%%%%%%%%%%%%%%%%%%%%%%%%%%%%%
%%% Notations
\subsection*{Notation and conventions} We work in the category of \textbf{schemes} of finite type over the complex numbers.  
A \textbf{curve} will be a reduced, connected, complete scheme  of pure dimension one (of finite type over $\mathbb C$). 
We will say a scheme $X$ has a \textbf{singularity of type $T$} at a point $x\in X$ if the completion of the local ring $\hat {\mathscr O}_{X,x}$ is isomorphic to the standard complete local ring with singularity of type $T$.  
The main focus will be on  curves with singularities of type $A_k$ ($k\ge 1$), $D_k$ ($k\ge 4$), $E_6$, $E_7$, and $E_8$.  A scheme  that is smooth or has singularities only of this type will be called an \textbf{ADE scheme}.
For example, we recall that  an $m$-dimensional $A_n$ singularity can be defined by the vanishing of a single polynomial of the form:
$$f_{A_n}= x_1^{n+1}+x_2^2+\ldots+x_{m+1}^2, \textrm{ for } n\ge 1.$$
Note that the index used in the notation of an ADE singularity is the Milnor number $\mu$ of the singularity (e.g. $\mu(E_7)=7$), and is known to be the dimension of a mini-versal deformation space of the singularity.

Two polynomials $f_1\in \mathbb C[x_1,\ldots,x_{m_1}]$ and $f_2\in \mathbb C[x_1,\ldots,x_{m_2}]$ are said to be \textbf{stably equivalent} (cf. \cite[p.187]{arnoldv1}) if  there exists an $m_3\ge \operatorname{max}(m_1,m_2)$ such that
$$
\frac{\mathbb C[[x_1,\ldots,x_{m_3}]]}{(f_1+x_{m_1+1}^2+\ldots+x_{m_3}^2)}\cong
\frac{\mathbb C[[x_1,\ldots,x_{m_3}]]}{(f_2+x_{m_2+1}^2+\ldots+x_{m_3}^2)}.
$$
Two hypersurface singularities are said to be stably equivalent  if they are defined by stably equivalent polynomials. 
A key fact is that the mini-versal deformation spaces (and the discriminants) for stably equivalent hypersurface singularities can be  identified (see also \S \ref{sdef3} below).

%%%%%%%%%%%%%%%%%%%%%%%%%%%%%%%%%%%%%%%%%
%%% S1 -- Deformations
\section{Versal deformations of ADE curves}\label{secdef} The purpose of this section is to briefly discuss the following standard result that describes the deformation space of projective varieties with at worst isolated lci singularities:

\begin{fact}\label{cor1}
Let $X$ be a projective scheme with  singular locus consisting of exactly $n$ isolated local complete intersection (lci) singularities $x_1,\ldots,x_n$ of type $T_1,\ldots,T_n$ respectively.  If 
\begin{equation}\label{eqnmaincond}
h^2(X,\mathscr O_X)=h^2(X,\mathscr Hom(\Omega_X,\mathscr O_X))=0,
\end{equation} 
then $X$  admits a smooth mini-versal deformation space $B_X$, with divisorial discriminant $\Delta_X$.  
Moreover, setting
$B_{T_i}$ to be  a mini-versal deformation space of the singularity at $x_i$ for each $i$, then \'etale locally, 

\begin{equation}\label{eqnloc1}
B_X\cong_{\textnormal{\'et}} B_{T_1}\times \ldots \times B_{T_n}\times \mathbb A_{\mathbb C}^m
\end{equation}
for some $m$, and under this identification,  setting $\pi_i$ to be  the  projection onto the factor $B_{T_i}$, 

\begin{equation}\label{eqnloc2}
\Delta_X= \pi_1^*\Delta_{T_1}+ \ldots + \pi_n^*\Delta_{T_n}.
\end{equation} 
\end{fact}

For reduced curves the two conditions in \eqref{eqnmaincond} are automatically satisfied. Additionally, we are interested in the very restricted case where $X$ has at worst ADE (and thus lci) singularities. Consequently, the above result applies for $X$ an ADE curve. In this situation, we will denote by $\sX\to B_X$ and $\Delta_X$ the mini-versal deformation of $X$ and the discriminant respectively.
For completeness, we briefly recall below the basic ingredients of Fact \ref{cor1}.

\subsection{Existence of mini-versal deformations}\label{sdef1}  The literature on deformations is extensive; we refer the reader to Sernesi \cite{sernesi} as a compendium of most of the standard results cited in this section.  
To begin, a general result in deformation theory (e.g. \cite[Corollary 2.4.2]{sernesi}) states that for a projective scheme $X$ with at most isolated singularities, the deformation functor $\operatorname{Def}_X$ has a semi-universal formal element (see \cite[Definition 2.2.6]{sernesi}); this is sometimes referred to more descriptively as a minimal formally-versal formal deformation.  Via a theorem of Grothendieck \cite{groth} (e.g. \cite[Theorem 2.5.13]{sernesi}), the condition that $h^2(X,\mathscr O_X)=0$ implies that every formal deformation of $X$ is effective  (see \cite[Definition 2.5.10]{sernesi}).  Artin's algebraization theorem \cite{aa} (e.g.  \cite[Theorem 2.5.14]{sernesi}) then implies that $X$ has a formally semi-universal  algebraic  deformation  (see \cite[Definition 2.5.9]{sernesi}), which is to say, $X$ has a mini-versal deformation space. We conclude that \emph{a projective scheme with at most isolated singular points, and $h^2(X,\mathscr O_X)=0$,  admits a mini-versal deformation space.}

\begin{rem}\label{remexistvers}
For curves, another argument is possible via Hilbert schemes, following the approach taken in Koll\'ar \cite[Theorem II 1.11]{kollar}. One obtains stronger algebraic properties this way than by applying the general theory outlined above.  
\end{rem} 

\subsection{Local versus global deformations}\label{sdef2}  
The global deformations of a reduced scheme $X$ are related to the deformations of the singularities by the local-to-global spectral sequence for $\mathrm{Ext}$ (see Sernesi \cite[Thm. 2.4.1(iv), Proposition 2.3.6]{sernesi}). In the situation of isolated singularities, this is easy to describe. 

Namely, let $X$ be a reduced, projective, lci scheme with singular locus consisting of exactly $n$ isolated singularities $x_1,\ldots,x_n\in X$ of type $T_1,\ldots,T_n$ respectively.  For the deformation functor $\operatorname{Def}_X$ there is a deformation obstruction theory with tangent space $\operatorname{Ext}^1(\Omega_X,\mathscr O_X)$ and obstruction space
$\operatorname{Ext}^2(\Omega_X,\mathscr O_X)$ (e.g. \cite[Theorem 2.4.1, Proposition 2.4.8]{sernesi}).  In addition, the deformation functor of the singularities, $\prod_{i=1}^n\operatorname{Def}_{X_{T_i}}$, has a deformation obstruction theory with tangent space given by $H^0(\mathscr Ext^1(\Omega_X,\mathscr O_X))$ and obstruction space
$H^0(\mathscr Ext^2(\Omega_X,\mathscr O_X))$ (e.g. \cite[\S3]{sernesi}). It then follows easily from a local-to-global spectral sequence argument that \emph{if  $h^2(X,\mathscr Hom(\Omega_X,\mathscr O_X))=0$, then the natural map 
$\operatorname{Def}_{X}\to \prod_{i=1}^n\operatorname{Def}_{X_{T_i}}$ 
is smooth}. 
\subsection{Deformations of lci singularities}\label{sdef3}  
To conclude we note that for isolated 
lci  singularities the deformation theory is very well understood. We refer the reader to \cite[\S 7]{vistolici}, where it is shown that given  an isolated lci singularity $X$, there exists a mini-versal deformation $\pi:\mathscr X\to B_X$, which is a relative complete intersection morphism (see \cite[Definition D.2.1]{sernesi}) with smooth base $B_X$.  
It is not hard to show in this case that the discriminant $\Delta_X$ is a divisor.  Note that from the description in \cite{vistolici}, one can easily establish that two stably equivalent hypersurface singularities have isomorphic mini-versal deformation spaces (and discriminants).

%%%%%%%%%%%%%%%%%%%%%%%%%%%%%%%%%%%%%%%%%
%%% S2 -- Weyl covers
\section{Weyl covers and hyperplane arrangements}\label{sectcover}
As discussed in the previous section, to understand the deformations of an ADE curve, it suffices to understand the local deformations of these singularities. The purpose of this section is to recall and discuss the following well known fact (see \S\ref{secwey}  for some references):

\begin{fact}\label{fact21}
Let $X$ be an ADE singularity of type $T$.  Let $B_T$ be a mini-versal deformation of $X$ with discriminant $\Delta_T$.  Define $W_T$ to be the Weyl group of type $T$ and $R_T$ be the corresponding root system.  Then there exists a Galois cover $f:B_T'\to B_T$ with covering group $W_T$ and ramification locus $\Delta_T$ such that $f^*\Delta_T$ is an arrangement of hyperplanes determined by the root system $R_T$.  The hyperplanes are in one-to-one correspondence with the roots in $R_T$ considered up to $\pm 1$.
\end{fact}

Using  Fact \ref{cor1} and Fact \ref{fact21} we make the following definition that will be relevant in later sections.

\begin{dfn}[Weyl cover]\label{dfnwc}
Let $X$ be a projective  ADE scheme satisfying \eqref{eqnmaincond}, with exactly $n$ singularities $x_1,\ldots,x_n$ of type $T_1,\ldots,T_n$ respectively. Let $B_X$ be  a  mini-versal deformation with discriminant $\Delta_X$.  Define $W_X=\prod_{i=1}^n W_{T_i}$ to be the product of the groups associated to the singularities of $X$, and let $R_X$ be the corresponding root system. We will refer to $W_X$ (resp. $R_X$) as the \textbf{Weyl group (resp. root system) associated to $X$}.  Then there exists a Galois cover 
\begin{equation}\label{eqndefwc}
f:B_X'\to B_X
\end{equation}
 with covering group $W_X$ and branch  locus $\Delta_X$ such that $f^*\Delta_X$ is an arrangement of hyperplanes determined by the root system $R_X$. Furthermore, the hyperplanes are in one-to-one correspondence with the roots in $R_X$ considered up to $\pm 1$.  The morphism 
\eqref{eqndefwc} will be called the \textbf{Weyl cover associated to $X$}.
\end{dfn}

\subsection{The Weyl cover}\label{secwey}  In the case of surface singularities, 
Fact \ref{fact21} is due to Brieskorn \cite{briessing} (see also Tjurina  \cite{tjurina0}, Artin \cite{artinsurf}, Pinkham \cite[Th\'eor\`eme principal, p.188]{pink} and  Arnold et al. \cite[II Theorem 3.8]{arnold}).    The general case is obtained via stabilization of the singularity (i.e. the deformation spaces and discriminants of stably equivalent hypersurface singularities are naturally identified; thus the surface case suffices for ADE singularities).  

In concrete terms, the Weyl covering \eqref{eqndefwc} of type $T$ is given by Chevalley's Theorem,  
which says that the subalgebra of $W_T$-invariant polynomials is  a polynomial ring itself; i.e.  \eqref{eqndefwc} is given by 
$$
B'=\Spec\left( \mathbb C[x_1,\dots,x_n]\right)\to B=\Spec\left( \mathbb C[x_1,\dots,x_n]^{W_T}\right).
$$ 
A Weyl group is  a finite reflection group; the hyperplane arrangement $\sH$ is simply the set of reflection hyperplanes. It is clear that $\sum_{H\in \sH} H$ is the ramification divisor of $f:B'\to B$, and that $f_*\sum_{H\in \sH} H=\left(\sum_{H\in \sH} H\right)/W$ is the branch divisor.  Brieskorn's theorem (\cite{briessing}) asserts that this branch divisor is the discriminant in a mini-versal deformation space of a singularity of type $T$ (for an explicit discussion in the $A_n$ case see Section \ref{secandirect}).

\subsection{Root systems and hyperplane arrangements} There is associated to a Weyl group $W_T$ a corresponding root system $R_T$  spanning a real vector space $V^\vee$.  This defines a  hyperplane arrangement $\mathscr H_{\mathbb R}$ in $V$.  The hyperplanes are in one-to-one correspondence with the positive roots of $R$ via the assignment $\alpha \mapsto \alpha^\perp = H_{\pm \alpha}:=\{v\in V:  \alpha(v)=0\}$.  The hyperplane arrangement $\sH$ in the Weyl cover of the mini-versal deformation space of a singularity of type ADE is easily seen to be the complexification of  the real hyperplane arrangement associated to a root system $R_T$.  The hyperplane arrangements associated to the root systems of  Weyl groups are called  \textbf{Weyl arrangements} (or more generally Coxeter arrangements), and  are well studied in the literature (see esp. Orlik--Solomon \cite{coxeterhyp}).

\subsection{Parabolic sub-root systems and the  stratification of the Weyl arrangement}\label{sstratification} 
 In Section \ref{secmssch}, we will be interested in the possible intersections of hyperplanes from the arrangement $\mathscr H$ (or equivalently $\mathscr H_{\mathbb R}$) and the resulting stratification of the discriminant.  This type of  question has been considered, in the context of simultaneous resolution of surface singularities, for a broader class of singularities by Wahl \cite{wahlduke}.
 In the case of ADE singularities, the stratification can be easily described in terms of the corresponding root systems. 
 For this purpose, as is customary, we let $L(\sH)$ be the partially ordered set of  intersections of hyperplanes from $\sH$, with the order given by reverse inclusion.
The action of the Weyl group $W$ on the root system $R$ induces an action on $L(\sH)$ compatible with the ordering.

It is easy to see that by associating to $Z\in L(\sH)$ the subroot system $R_Z=R\cap Z^\perp$ one obtains a bijection between $L(\sH)$ and the set of parabolic  subroot systems of $R$ (e.g. \cite[VI Prop. 24]{bourbaki}).
 Clearly, the codimension of $Z$ is equal to the rank of $R_Z$. Furthermore, two strata $Z$ and $Z'$ are $W$-equivalent iff the corresponding subroot systems $R_Z$ and $R_{Z'}$ are conjugate by the Weyl group $W$($=W(R)$). Thus,  if $W_Z=W(R_Z)$ denotes the Weyl group associated to a stratum $Z\in L(\sH)$, then the number of strata $W$-equivalent to $Z$ is equal to 
$$|W:N_W(W_Z)|$$
where $N_W(W_Z)$ is the normalizer of $W_Z$ in $W$ (e.g. \cite[(3.4)]{coxeterhyp}). Explicit computations of the possible strata and their numbers in the case of root systems of small rank, including all exceptional cases, are worked out in Orlik--Solomon \cite[Tables I, VI-VIII]{coxeterhyp} (see also Carter \cite{carter}).

\begin{exa}\label{exawce6}
Consider the case of a hyperplane arrangement associated to an $E_6$ root system.   There is a unique zero dimensional stratum corresponding to the entire $E_6$ root system, and from  Table I on p. 284 of \cite{coxeterhyp}, the one dimensional strata are described as follows:  there are $27$ strata associated to root systems of type $D_5$, there are $36$ strata associated to root systems of type $A_5$, there are $216$ strata associated to root systems of type $A_1\times A_4$, and there are $360$ strata associated to root systems of type $A_1\times A_2\times A_2$. 
\end{exa}

Finally, we note the following combinatorial description of the strata $Z\in L(\sH)$ given in terms of Dynkin diagrams: 
\emph{there exists a 
stratum $Z\in L(\sH)$ associated to a parabolic subroot system $R'$ of type $T'$ if and only if there exists a full subdiagram of type $T'$ of the Dynkin diagram associated to $R$}.
  The ADE case, the only one we will need in this paper,  is due to  Grothendieck (see Demazure \cite[2.11 p.137]{demazure}, Arnold et al. \cite[Theorem p.132]{arnoldsum}).
Recall that a (possibly disconnected) sub-diagram $D'$ of a Dynkin diagram $D$ is said to be full if it is obtained from $D$ by removing some vertices  together with all of the edges terminating in those vertices. For example, for the $E_6$ hyperplane arrangement, the diagram:
$$
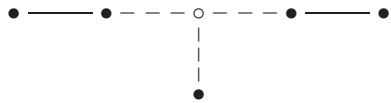
\xymatrix@R=.25cm{
{\bullet}\ar@{-}[r]&{\bullet}\ar@{--}[r]&{\circ}\ar@{--}[r]\ar@{--}[dd]&{\bullet}\ar@{-}[r]&{\bullet}\\
                                   &                   &                                              &  \\
                        &                   &{\bullet}                     &                   
}$$ 
shows that there is a dimension $1$ stratum $Z$  (unique up to the action of $W(E_6)$) corresponding to a subroot system $R_Z$ of type $A_1\times A_2\times A_2$. The number of strata conjugate to $Z$ in $L(\sH)$ is $|W(E_6):N_{W(E_6)}(W(A_1\times A_2\times A_2))|=360$ (e.g. \cite[Table 9]{carter}).

\subsection{Connection with deformations}\label{seccomb}
Let $\mathscr X_T\to B_T$ be the versal deformation of a singularity of type $T$, with Weyl cover $B_T'\to B_T$. It is well known that the stratification of the hyperplane arrangement $\sH$ as discussed above is related to  the singularities of the fibers as follows: let $z\in Z$ be a point in the stratum $Z$ not contained in any stratum of lower dimension.  The stratum $Z$ corresponds to a parabolic subroot system of type $R_Z$ if and only if the fiber of $B_T'\times_{B_T} \mathscr X_T$ over $z$ has singularities of type $R_Z$ (e.g. Pinkham \cite[Th\'eor\`eme 2 p.189]{pink}).  For example,  in the Weyl cover of the deformation space of an $E_6$ singularity, there will be $1$-dimensional strata corresponding to   parabolic subroot systems of type $A_2\times A_2\times A_1$. Over such a stratum, the corresponding fibers will have as singularities precisely two cusps and one node.

%%%%%%%%%%%%%%%%%%%%%%%%%%%%%%%%%%%%%%%%%
%%% S3 -- Resolution of the moduli map
\section{Wonderful blow-ups and morphisms to the moduli scheme $\overline M_g$}\label{secmssch}
Given a versal deformation $\mathscr X\to B_X$ of a curve with  ADE  singularities, one obtains a natural rational moduli map $B_X\dashrightarrow  \overline{M}_g$. In this section we construct an explicit resolution of this map (see Theorem  \ref{teoresmm}), with certain canonical properties. The two main ingredients used in the construction are the ``wonderful'' blow-up  of de Concini-Procesi \cite{dcp} and  an extension theorem of de Jong-Oort \cite{dejongoort}.  

\subsection{Wonderful blow-ups and the log alteration of the discriminant} The discriminant $\Delta_X$ plays a  key role in resolving the moduli map $B_X\dashrightarrow  \overline{M}_g$. As noted in the previous section,  by passing to a finite cover $B'_X$, we can assume that the discriminant is a hyperplane arrangement.  At this point we can employ 
the \emph{wonderful blow-up} of  de Concini-Procesi \cite{dcp},  a  procedure, with some canonical properties, for
replacing the hyperplane arrangement by a divisor in a simple normal crossings configuration.  We will follow here the presentation given in   MacPherson-Procesi \cite{macpherson} as it is more suited to our situation of hyperplane arrangements associated to root systems.

Recall that we are considering an arrangement of hyperplanes $\mathscr H$  in a vector space $V\otimes_{\mathbb R}\mathbb C=V_{\mathbb C}$  determined by a root system $R\subset V^\vee$.  The hyperplane arrangement determines a stratification of $V_{\mathbb C}$ with strata corresponding to parabolic subroot systems (see \S\ref{sstratification}).  We will say that a stratum is an \textbf{irreducible stratum} if it corresponds to an irreducible (parabolic) sub-root system of $R$.  One easily checks that, in the notation of MacPherson-Procesi \cite{macpherson}, this stratification is a \emph{conical} stratification of $V_{\mathbb C}$, and the irreducible strata defined here correspond to the irreducible conical strata in their notation.  

The \textbf{(minimal) wonderful blow-up} 
$\tilde {V}_{\mathbb C}$ of $V_{\mathbb C}$ associated to the arrangement  is defined to be the result of repeatedly blowing up the (closure of the strict transform of)  irreducible strata of minimal dimension, until they are all  of codimension $1$ (see  \cite[Def. p.132]{macpherson}). The key results we will use are:

\begin{enumerate}
\item   The wonderful blow-up is well defined up to   isomorphism \cite[Prop. 2 on p.132]{macpherson}.
\item  The wonderful blow-up is smooth   \cite[Prop. on p. 131]{macpherson}.
\item The pull-back of the hyperplane arrangement is an SNC divisor 
\cite[Prop. 1 on p. 132]{macpherson}; we will refer to it also as the \emph{boundary divisor}.
\item  The irreducible components of the boundary divisor of the wonderful blow-up are 
indexed by the irreducible strata associated to the arrangement.  Moreover, each irreducible divisor is the closure of the inverse image of the corresponding irreducible stratum  \cite[Thm. (1) on p. 134]{macpherson}. 
\end{enumerate}
In short, the wonderful blow-up is a log resolution of the pair $(V_{\mathbb C},\mathscr H)$. Furthermore, in our situation,  the components of the boundary divisor are indexed by the irreducible parabolic  sub-root systems $R'$ of $R$. 

\begin{rem}
 We emphasize that there is no assumption that $R$ be irreducible in the wonderful blow-up construction.   Moreover, the cases where $R$ is reducible can be easily understood in terms of the irreducible cases (see the discussion of ``product stratification'' from \cite[p. 126--127]{macpherson}).
 \end{rem}

Returning to our situation, deformations of ADE curves, we define:

\begin{dfn}[Canonical log alteration]\label{dfnclr}
Let $B_X$ be a mini-versal deformation space of a projective ADE scheme $X$ satisfying \eqref{eqnmaincond}, with discriminant $\Delta_X$.  The  \textbf{canonical log alteration} of $(B_X,\Delta_X)$ is the wonderful blow-up of the  Weyl cover $B'_X$ of $B_X$ (w.r.t. the root system $R_X$ and the corresponding Weyl arrangement).  We will denote this space and the pull-back of the discriminant by $\tilde B_X$ and $\tilde \Delta_X$ respectively.
\end{dfn}

We conclude:

\begin{pro}\label{proresdisc}
Let $B_X$ be a mini-versal deformation space of a projective  ADE  scheme $X$ satisfying \eqref{eqnmaincond}, and let $\Delta_X\subseteq B_X$ be the discriminant.   
The canonical log alteration
\begin{equation}\label{eqnres}
(\tilde B_X,\tilde \Delta_X) \to (B_X', \Delta_X')\to (B_X,\Delta_X)
\end{equation}
of the pair $(B_X,\Delta_X)$ consists of a smooth variety $\tilde B_X$ and an SNC divisor $\tilde \Delta_X$.  The composite map $\tilde B_X\to B_X$ is proper, generically finite,
and $W_X$-equivariant with respect to the natural $W_X$-action on $\tilde B_X$. 

The irreducible components of $\tilde \Delta_X$ are indexed by the irreducible parabolic sub-root systems of $R_X$, and consequently,  the generic point of an irreducible component of $\tilde \Delta_X$ corresponds to a projective ADE scheme  with a unique singularity. 
\end{pro}
\begin{proof}
Since the stratification of $\sH$ is $W_X$-invariant, one easily checks that the $W_X$-action on $B'_X$ extends to $\tilde B_X$. The statement then follows from the above discussion.
\end{proof}

\begin{exa}\label{exawbe6}
Consider the case of a mini-versal deformation space $B_X$ of an $E_6$ singularity (see Example \ref{exawce6}), and the associated root system $R_X$ of type $E_6$.  The wonderful blow-up of the Weyl cover $B_X'$ will start with the blow-up of the origin in $B'_X$ (corresponding to the trivial irreducible subroot system $R'=R_X$), and then continue with the blow-up of the $1$-dimensional irreducible strata (which, up to $W$-equivalence, are associated to subroot systems $R'$ of type $A_5$ and $D_5$), and so on, until concluding with the  blow-up of the codimension $2$ strata (corresponding to subroot systems $R'$ of type $A_2$).  At the end of the procedure, the boundary divisor of the canonical log alteration will contain $36$ irreducible components of type $A_1$, $120$ of type $A_2$, $270$ of type $A_3$, $216$ of type $A_4$, $45$ of type $D_4$, $36$ of type $A_5$, $27$ of type $D_5$ and $1$ of type $E_6$ (see \cite[Table I]{coxeterhyp}). Note that for the $E_6$ case, all the divisors of a certain type are conjugate by the Weyl action. 
\end{exa}

%%% de Jong-Oort Theorem
\subsection{Morphisms to the moduli scheme of curves}
Once the discriminant of a family of genus $g$ curves is in simple normal crossing position, one obtains an extension of the moduli map. Namely, this is precisely the content of an extension theorem of  de Jong-Oort \cite{dejongoort} (see also Cautis \cite[Theorem 1.2, Theorem 4.1]{cautis}):
\emph{Let $S$ be a smooth scheme, and $\Delta$ an effective SNC divisor on $S$.  Set $S^\circ=S\setminus\Delta$, and suppose 
there is a morphism $\phi^\circ:S^\circ \to \mathcal M_g$ (i.e. a family of smooth curves over $S^\circ$). Then there is a morphism $\phi:S\to \overline{M}_g$ extending $S^\circ\to \overline{\mathcal M}_g\to \overline M_g$}. 

Returning to the case of an ADE curve $X$ with $p_a(X)=g$, since $B_X\setminus\Delta_X$ comes equipped with a family of smooth curves, there is rational map $B_X\dashrightarrow \overline M_g$. Also, as noted above in Proposition \ref{proresdisc}, we can take the quotient
$$(\overline B_X, \bar \Delta_X)=(\tilde B_X,\tilde \Delta_X)/W_X.$$ 
Clearly, $\overline B_X\to B_X$ is proper and birational. Applying the  de Jong-Oort extension theorem, we then conclude: 

\begin{teo}\label{teoresmm} Let $X$ be an ADE curve.
In the notation above, the rational map $(\overline B_X,\bar \Delta_X)\dashrightarrow \overline M_g$ extends to a morphism.
In other words,  we have a diagram 
\begin{equation} 
\xymatrix
{ 
(\tilde B_X,\tilde \Delta_X)\ar@{->}[d]_{\textnormal{Blow-up}}\ar@{->}[r]^{/W_X}&(\overline B_X,\bar \Delta_X)\ar@{->}[d]_{\textnormal{}}\ar@{->}[rd]&\\
(B_X',\Delta_X')\ar@{->}[r]^{/W_X}&(B_X,\Delta_X)\ar@{-->}[r]&(\overline M_g,\Delta)
}
\end{equation}
giving an explicit resolution of  the rational map $(B_X,\Delta_X)\dashrightarrow (\overline M_g,\Delta)$.
\end{teo}

\begin{proof}
Since $\tilde \Delta$ is SNC, the existence of a morphism $\tilde B_X\to \overline M_g$  follows from the result of  Jong-Oort \cite{dejongoort} mentioned above (see also \cite[Theorem 1.2]{cautis}).   
The map $\tilde B_X\to \overline B_X$ is finite; in fact it is defined by taking the quotient of a normal (smooth) quasi-projective variety by a finite group, and consequently $\overline B_X$ is also normal and quasi-projective.  Since $\overline M_g$ is a complete variety,  it follows that the rational map $\overline B_X\dashrightarrow \overline M_g$ extends to a morphism (e.g. Cautis \cite[Lemma 2.4]{cautis}).
\end{proof}

Let $\overline{\mathcal M}_g^{ADE}$ be a separated (not necessarily complete) Deligne-Mumford stack  parameterizing some class of $ADE$ curves of genus $g$.  We refer the reader to Smyth \cite{smyth09} for more details on such stacks in general.
 Since the versal deformations give local presentations of  such a stack, we conclude:
 
\begin{cor}
Suppose that  $\overline{\mathcal M}_g^{ADE}$  admits a coarse moduli space $\overline{M}_g^{ADE}$.
Let $X \in \overline{\mathcal M}_g^{ADE}$ be an ADE curve of genus $g$ with  (finite) automorphism group $\operatorname{Aut}(X)$.  Let $B_X$ be a mini-versal deformation space of $X$ with discriminant $\Delta_X$.  
The \'etale neighborhood $B_X/\operatorname{Aut}(X)$ of $[X]\in \overline M_g^{ADE}$ admits a rational map $$B_X/\operatorname{Aut}(X)\dashrightarrow \overline M_g$$ that is resolved by a weighted blow-up obtained as the successive quotient of the canonical log alteration of $(B_X,\Delta_X)$ by the action of the Weyl group $W_X$ and the automorphism group $\operatorname{Aut}(X)$ respectively.  In other words, in the notation of the theorem, we have a diagram 
\begin{equation} 
\xymatrix
{ 
(\overline B_X/\operatorname{Aut}(X),\bar \Delta_X/\operatorname{Aut}(X))\ar@{->}[d]_{\textnormal{}}\ar@{->}[rd]&\\
(B_X/\operatorname{Aut}(X),\Delta_X/\operatorname{Aut}(X))\ar@{-->}[r]&(\overline M_g,\Delta)
}
\end{equation}
\end{cor}

\begin{proof}
This follows directly from Theorem  \ref{teoresmm} and another application of \cite[Lemma 2.4]{cautis}.
\end{proof}

\begin{rem}
By a weighted blow-up $\tilde{Y}\to Y$ we mean the quotient of a standard blow-up $\tilde X\to X$, with smooth center $Z\subseteq X$, by a finite a group action $G$. Note that $\tilde{Y}=\tilde{X}/G$ may be singular (with finite quotient, singularities), even in the case where  $Y=X/G$ and the center $Z/G$ are smooth.  As an example, the standard weighted blow-up $\tilde{Y}$ of the origin in $Y=\bC^n$ with weights $(w_1,\dots,w_n)$ can be described as the quotient of the standard blow-up $\tilde{X}$ of the origin in $X=\bC^n$ by the group $G=\mu_{w_1}\times\dots\times \mu_{w_n}$. Equivalently, $\tilde{Y}$ is the blow-up of $Y$ along  the ideal $I=( x_1^{w_1},\dots,x_n^{w_n})$. The exceptional divisor here is the weighted projective space $\bP(w_1,\dots,w_n)$. However, note that our notion of a weighted blow-up is slightly more general than the usual notion of weighted blow-up from toric geometry (e.g. the stabilizers involved need not to be abelian). 
\end{rem}

%%%%%%%%%%%%%%%%%%%%%%%%%%%%%%%%%%%%%%%%%
%%% S4 -- Related Work
\section{Connections with other work}
\subsection{The $0$-dimensional analogue and the Fulton-MacPherson compactification of configuration space}\label{confspace}
 Theorem  \ref{teoresmm} has a well known analogue in the zero-dimensional case.  Namely, one considers the configuration space $$\mathcal C_{0,n}:=\{(z_1,\dots,z_n)\in\bC^n\mid z_i\neq z_j \textrm{ for } i\neq j\}$$ of $n$ (labeled) points on the complex line. The configuration space has a naive compactification given by $\bC^n$, as well as a modular compactification constructed by Fulton-MacPherson \cite{fulton} (in much more generality). In this situation, the Fulton-MacPherson compactification coincides with the wonderful blow-up of $\bC^n$ corresponding to the arrangement of hyperplanes given by the diagonals $(z_i=z_j)$ (see \cite[\S5.1]{macpherson}). 

By compactifying  $\bC\subset \bP^1$, one gets a natural map $\mathcal C_{0,n}\to M_{0,n}$, and then a rational map $\bC^n \dashrightarrow \overline{M}_{0,n}$. The Fulton-MacPherson compactification of $\mathcal C_{0,n}$ is a minimal resolution of this rational map. In fact, this is also closely related to Kapranov's construction of $\overline{M}_{0,n}$ (see \cite{kap1} and \cite[p. 473]{thaddeus}) and some earlier work of Keel \cite{keel}. 

Finally, we note that the Weyl cover $B'$ of the versal deformation space of the $A_{n-1}$ singularity ($n$ points colliding) 
is naturally identified with the hyperplane $(z_1+\dots+z_n=0)$ in $\bC^n$. This hyperplane section is transverse to the blow-ups occurring in the Fulton-MacPherson construction. It follows that  the wonderful blow-up $\tilde B$ of $B'$ in the case of an $A_{n-1}$ root system is the restriction to $B'$ of the wonderful blow-up used in the Fulton-MacPherson construction. In particular, $\tilde B$ maps to $\overline{M}_{0,n}$. 

\subsection{The image of the extended moduli map and results of Hassett \cite{hassettstable}}\label{sechassettstable} By Theorem \ref{teoresmm}, there exists a map $\bar \varphi:\bar B_X\to \overline {M}_g$.  The discriminant $\bar \Delta_X$ is naturally stratified by strata $Z_{R'}$ corresponding to parabolic subroot systems $R'\subseteq R_X$ considered up to conjugacy by $W_X$. We now discuss the image in $\overline {M}_g$ of the various strata of $\bar \Delta_X$, or equivalently   the image of the associated strata in $\tilde \Delta_X$.

To begin, we make two basic observations. First, it suffices to understand the image of the generic point of a stratum under  $\bar \varphi$ (i.e. the image of the stratum will be contained in the closure of the image of the generic point). Second, we can assume $R'$ is irreducible (i.e. the divisorial case), otherwise the stratum will be an intersection of higher dimensional strata and the image can be understood inductively. 

We are thus reduced to understanding the image of 
a generic point in a boundary divisor $D_{R'}\subset \tilde B_X$ corresponding to an irreducible (parabolic) subroot system $R'\subset R_X$. This question was answered previously by Hassett \cite{hassettstable}. Namely, let $\tilde{b}\in D_{R'}$ be a general point with image $b\in B_X$. By the discussion of \S\ref{seccomb}, we know  that the fiber $X_b$ over $b\in B_X$ in the versal deformation is a curve with a unique singularity of type $R'$.  Considering an arc through $\tilde b\in \tilde B_X$, the stable curve $C_m$ corresponding to $m=\bar \varphi(b)\in \overline M_g$ is the result of the stable reduction of a one-parameter family  with central fiber $X_b$. 

The curve $C_m$ will be the stable model of the nodal curve that consists of two components: the normalization $\widetilde{X}_b$ of $X_b$ and a curve $T$  called a tail.  The topological type of the tail and the attaching data is determined by $R'$,  while the modulus is determined by the choice of one-parameter degeneration (in this case, since $\bar \varphi$ is a morphism, the choice of $\tilde b$ determines the tail).  For further discussion, we point the reader to Hassett \cite{hassettstable}, where a  complete description of the tails and attaching data for ADE singularities is given.

\begin{exa}\label{exaansr}
As an example, consider the case where the curve $X_b$ has a unique singularity of type $A_{n}$. The tail $T$ is a hyperelliptic curve of genus $\lfloor n/2\rfloor$  attached to the normalization $\widetilde{X}_b$ in one or two points depending on the parity of $n$. These points are the special points of the normalization and are either a Weierstrass point or conjugate points for the tail if $n$ is even or odd respectively (see also \S \ref{secandirect}).  Denote by $C_m$ the stable model.  Let us now assume that $\tilde b$ is general, and that $\tilde X_b$ and its special point(s) have general moduli.  Then  in the $A_{2n}$ case, $\operatorname{Aut}(C_m)\cong \mathbb Z/2\mathbb Z$, whereas in the $A_{2n+1}$ case, $\operatorname{Aut}(C_m)$ is trivial     (unless $g=n+1$, in which case $\operatorname{Aut}(C_m)\cong\mathbb Z/2\mathbb Z$).
The generator of the non-trivial group is the hyperelliptic involution of the tail, which fixes the Weierstrass point of attachment, or the node, in the respective cases.
\end{exa}

\subsection{Relation to the Hassett-Keel program}
Finally, to relate our construction to the Hassett-Keel program, we denote by $\mathcal S_{R'}\subset B_X$ the stratum corresponding to  $R'$, $E_{R'}$ the pre-image of $\mathcal S_{R'}$ in $\overline B_X$, and $\mathcal T_{R'}\subset \overline{M}_g$ the locus of tails in $\overline{M}_g$.
To be more precise, $\mathcal T_{R'}$  is the locus in $\overline M_g$ corresponding to stable curves that have a tail attached in the way described by Hassett \cite{hassettstable}.
 Theorem \ref{teoresmm} gives a diagram: \begin{equation}\label{hkdiag} \xymatrix {  \ \ \ \ \ \ \ \ E_{R'}\subset \overline B_X\ar@{->}[d]\ar@{->}[rd]^{\bar\varphi}&\\ \ \ \ \ \ \ \ \ \mathcal S_{R'}\subset B_X\ar@{-->}[r]&\mathcal T_{R'}\subset \overline{M}_g } \end{equation} which is consistent with the predictions of the Hassett-Keel program in the following manner. The program predicts that for $\overline M_g(\alpha):=\operatorname{Proj} R(\overline M_g,K_{\overline M_g}+\alpha \Delta)$, as $\alpha\to 0$ the different ``tail loci'' will be  replaced by ``singularity loci'' (with the difference of dimensions accounted for by the moduli of the \emph{crimping data}
 of the singularities; see van der Wyck  \cite{vdwyck} for the definition of crimping data).
 
  While the global Hassett-Keel program is only known for $\alpha=9/11$ where $\Delta_1$ is contracted and the new locus corresponds to the cuspidal locus, and for $\alpha=7/10$ where the elliptic bridge locus is flipped to the tacnodal locus (see \cite{hh,hh2}), the local picture always holds (see diagram (\ref{hkdiag})). Furthermore, the order of blow-ups in the wonderful blow-up is consistent with the expectations of the  Hassett-Keel program; i.e. the $A_2$ loci are blown-up last, preceded by the $A_3$ loci, etc.

\subsection{Results of Fedorchuk}  
For the case of $A$ and $D$ singularities, Fedorchuk \cite{fedorchuk} has recently given a description of the space $\overline B_X$ resolving the rational map $B_X\dashrightarrow \overline M_g$ in terms of moduli spaces of hyperelliptic curves.  We refer the reader to \cite[\S 7]{fedorchuk} for a more detailed description of the relationship between the spaces.

%%%%%%%%%%%%%%%%%%%%%%%%%%%%%%%%%%%%%%%%%
%%% S5 -- The monodromy obstruction
\section{ADE monodromy representations in odd dimensions}\label{secmono}
This section concerns the monodromy representation for the versal deformation $\sX/B_X$ of an ADE curve $X$.  After some topological preliminaries concerning the complement of the discriminant $B_X\setminus \Delta_X$, we establish the main results of the section (Thm. \ref{teomonvan} and Cor. \ref{teomono}) stating that the monodromy along generic loops around $A_{2n}$ divisors in $\widetilde{B}_X$ is not unipotent. 
Due to the topological considerations involved, in this section we work in the analytic category.  Most  of the discussion is valid in higher dimensions as well; in fact, the discussion is most natural when connected with the case of surfaces.   The main point is that under stabilization, the deformation spaces  of singularities of different dimensions and the corresponding discriminants can be  identified. Furthermore, the monodromy representations associated to stably equivalent singularities are closely related  by Thom-Sebastiani type results (e.g. \cite[Ch. 2 \S1.7]{arnoldsum}).

\subsection{Artin groups and the topology of the complement of a Weyl hyperplane arrangement} The first step in understanding the monodromy representation is to get a hold on the fundamental group of the complement of the  discriminant $B_X\setminus \Delta_X$. In the case at hand, the versal deformations of ADE singularities, the situation is  well understood via the work of Brieskorn \cite{brieskorndiefund,brieskornartin} and Deligne \cite{deligneartin} (see also \cite[\S5.3]{arnoldsum}). In particular, we recall the well known facts:
\begin{itemize}
\item[i)] The fundamental group of the complement of the discriminant $\Delta_X$  is isomorphic to the Artin group associated to the Coxeter diagram $\Gamma_X$ of $R_X$:
\begin{equation}\label{pi1artin}
\pi_1(B_X\setminus \Delta_X)\cong \Art(\Gamma_X);
\end{equation}

\item[ii)] The restriction of the ramified cover $B_X'\to B_X$ to the complement of the discriminant is a covering space, corresponding to the exact sequence:
\begin{equation}\label{covergp}
1\to \pi_1(B_X'\setminus \Delta_X')\to \pi_1(B_X\setminus \Delta_X)\to W(R_X)\to 1.
\end{equation}
\end{itemize}

A little less known is Proposition \ref{promain} below, which expresses a general simple loop $\sigma\in \pi_1(B_X'\setminus \Delta_X')$ around the origin in $B_X'$ in terms of the standard generators of the Artin group $\Art(\Gamma_X)$. To state and prove the proposition we need to recall in \S\ref{sartin} some facts about Artin groups  and in \S\ref{sloop} some facts about the topology of the complement of Weyl hyperplane arrangements. 

%%% Artin groups
\subsubsection{Artin and Coxeter groups}\label{sartin}  Let $\Gamma$ be a  weighted graph  with vertices indexed by a set $I$ and with an edge for each pair of vertices $\{i,j\}$ with $i\ne j$  labeled with an  integral weight $m_{ij}\ge 2$ (and possibly $\infty$). 
To the graph $\Gamma$, one  associates two groups: {\it the Artin group} and {\it the Coxeter group} respectively. The Artin group $\operatorname{Art}(\Gamma)$ is the group generated by $\{t_i\}_{i\in I}$ subject to the ``braid'' relations: 
\begin{equation}\label{braid}
t_it_jt_i \ldots=t_jt_it_j\ldots
\end{equation}
where both sides have $m_{ij}$ letters. The associated Coxeter group $W(\Gamma)$ has the generators $\{s_i\}_{i\in I}$,  satisfying the same ``braid'' relations, with the additional condition that each generator is an involution (i.e. $s_i^2=1$). In particular, note that there is a natural epimorphism $\operatorname{Art} (\Gamma) \to W(\Gamma)$.

We are interested in the particular case when $W(\Gamma)$ is a  Weyl group of type ADE; i.e. $\Gamma$ is a Coxeter graph of type $T_n\in \{A_n,D_n,E_n\}$.   
Note that since we have fixed a set of generators $\{t_i\}_{i\in I}
$ for $\operatorname{Art} (\Gamma)$ and their projections $s_i$ in $W(\Gamma)$  there is a given Weyl chamber in the real vector space $V$, where $V^\vee$ contains  the associated root system $R(\Gamma)$.   We will then use the following notation.  We will set  $w_0\in W(\Gamma)$ to be the longest element (determined by the choice of Weyl chamber).  The product of the generators $s_1\cdots s_n\in W(\Gamma)$, the so-called Coxeter element of $W(\Gamma)$, will have its order (the Coxeter number)  denoted by $h$. Recall that $h$ is even, except in the $A_{2n}$ case.
The  product $\Pi=t_1\cdots t_n\in \operatorname{Art}(\Gamma)$ of the generators will be called the \textbf{Artin-Coxeter element}.
We set $\mathfrak D\in \operatorname{Art}(\Gamma)$ to be  the Garside element   (N.B. $\mathfrak D$ is a specific lift of $w_0\in W(\Gamma)$ \cite[Satz 5.6, Proposition 5.7]{brieskornartin}).   The key point is   a result of Brieskorn-Saito \cite[Lemma 5.8]{brieskornartin} stating that  
$\mathfrak D^2=\Pi^h$ (up to conjugacy).  In fact, excepting the  $A_{2n}$ case, one has $\mathfrak D=\Pi^{h/2}$ (up to conjugacy).

%%% Loop around the origin
\subsubsection{Topology of the complement of a Weyl hyperplane arrangement}\label{sloop}
We now focus on the case where $\Gamma$ is a (possibly disconnected) ADE Dynkin diagram of type $T$.  Let $W=W(\Gamma)$, $\operatorname{Art}=\operatorname{Art}(\Gamma)$, and $V^\vee$ be the real vector space containing the root system $R=R(\Gamma)$.
Let $C\subseteq V$ be the Weyl chamber associated to the generators of $W$ determined by $\Gamma$.  
Let $B'=V\otimes_{\mathbb R}\mathbb C$, $B=B'/W$, and $\Delta'=\mathscr H=\bigcup_{\alpha\in R_X} (H_{\alpha}\otimes_{\mathbb R}\mathbb C)$ be  the associated hyperplane arrangement.
In other words, we are considering the hyperplane arrangement associated to an ADE  curve $X$ with singularities of type $T$ (Definition \ref{dfnwc}).  
Let
$
(B')^\circ=B'\setminus \Delta' 
$
be the complement of the hyperplane arrangement.  Similarly, set $B^\circ=(B')^\circ/W$; i.e. the complement of the discriminant. Note that the longest element $w_0\in W$  is also distinguished by the fact that it sends $C$ to $-C$.
Identifying $B'=V\otimes_{\mathbb R}\mathbb C$ as $V\oplus \sqrt{-1}V$, we have $C\subseteq B'$.  

We are interested in understanding the fundamental groups of $B^\circ$ and  $(B')^\circ$. Thus, fix a base point
$\ast'\in C\subset (B')^\circ$  with image 
 $ \ast\in B^\circ$.  As described in Looijenga \cite[p.195]{looijengaartin}, there exist
contractible (analytic) open subsets $\mathbb U^+\subseteq (B')^\circ$ (resp. $\mathbb U^-\subseteq (B')^\circ$) such that for each $w\in W$, we have  $w(\ast')\in w\cdot C\subset \mathbb U^+\cap \mathbb U^-$.  (In Looijenga's notation, $\mathbb U^+=\mathbb U$ and $\mathbb U^-=-\mathbb U$.)    
Thus for each $w\in W$, there is a path $\gamma_w^+$ (resp. $\gamma_w^-$)  contained in $\mathbb U^+$ (resp. $\mathbb U^-$) connecting $\ast'$ to $w(\ast')$, unique  up to an end-point fixing homotopy.   Define maps 
$$t^+:W\to \pi_1(B^\circ,\ast)  \ \ \text{ and } \ \ t^-:W\to \pi_1(B^\circ, \ast)
$$ 
by sending $w$ to the class of the image of the path $\gamma^+_w$ (resp. $\gamma_w^-$). Brieskorn's theorem \cite[Zusatz p.58]{brieskorndiefund} can then be stated as follows (see also \cite[Prop. 2.1]{looijengaartin}): 
\emph{There is an isomorphism
\begin{equation}\label{eqnagiso}
\operatorname{Art}\to \pi_1(B^\circ,\ast)
\end{equation}
induced by the assignment $t_i\mapsto t^+(s_i)=\left(t^-(s_i)\right)^{-1}$.  Under this isomorphism $t^+$ and $t^-$ are sections (as maps of sets) of the homomorphism $\operatorname{Art}\to W$,    
and \begin{equation}\label{eqnmlgarside}
t^+(w_0)=\left(t^-(w_0)\right)^{-1}=\mathfrak D,
\end{equation}
the Garside element}.
With this we have the following proposition:

\begin{pro}\label{promain}
Let $\sigma\subset (B')^\circ$ be a small simple loop around the origin based at the point $\ast'\in (B')^\circ$ and lying in the complexification of a line in $V$ spanned by the origin and the point $\ast'$.  
The push-forward map, together with the isomorphism \eqref{eqnagiso} give an inclusion
$$
\pi_1((B')^\circ,\ast')\hookrightarrow \pi_1(B^\circ,\ast)\cong \operatorname{Art}
$$
into the Artin group.
Under this identification, and up to conjugacy,
$$
\sigma=\Pi^h
$$
where $\Pi$ is the Artin-Coxeter element and $h$ is the Coxeter number.
\end{pro}
\begin{proof} 
From the description in \cite{looijengaartin}, it is not hard to see that the sets $\mathbb U^\pm $ have the following property.   Let $L\subseteq V$ be a one dimensional linear subspace such that $L\cap C\ne \emptyset$.  Consider the complexification $L_{\mathbb C}:=L\otimes_{\mathbb R} \mathbb C\subseteq  V \otimes_{\mathbb R} \mathbb C$.  Then $L_{\mathbb C}\cap \mathbb U^\pm$ can be identified with  the complex line minus a cut along the non-negative (resp. non-positive) imaginary axis.  In particular, a small circle $\sigma \subset L_{\mathbb C}$ centered at the origin and with base point $\ast' \in L$ can  be broken down into semi-circles $\sigma^+\subset \mathbb U^+$ and $\sigma^-\subset \mathbb U^-$ with $\sigma = (\sigma^-)^{-1}\sigma^+$.  The common endpoint of these paths is in $V$, and in fact it lies in $-C$. The same is true for $w_0(\ast')$.  As $-C$ is a cone, there is a line segment $\delta$ in $-C$ joining those points.  Then we have $\delta\sigma^+$ lies in $\mathbb U^+$, while $\delta \sigma^-$ lies in $\mathbb U^-$.  Thus, from \eqref{eqnmlgarside} of Brieskorn's theorem, we have $\delta\sigma^+=t(w_0)=\mathfrak D$ and $\delta\sigma^-=t^-(w_0)=\mathfrak D^{-1}$.  
In conclusion we have $\sigma=(\sigma^-)^{-1}\delta^{-1}\delta\sigma^+=\mathfrak D^2$, which up to conjugacy is equal to $\Pi^h$
 by \cite[Lemma 5.8]{brieskornartin}. 
\end{proof}

%%% the monodromy action
\subsection{Monodromy for families of varieties with ADE singularities}\label{smono} 
We now consider $X$ a projective  ADE  variety of dimension $d$ satisfying  \eqref{eqnmaincond}  (for instance an ADE curve)  together with a mini-versal deformation $\sX/B_X$. With notation as above (e.g. $B^\circ=B_X\setminus \Delta_X$ and $\ast\in B^\circ$ is a base point), we are interested in understanding the monodromy representation:
$$\rho:\pi_1(B^\circ, \ast)\to \operatorname{Aut}\left(H^d(X_{\ast},\mathbb Z)\right),$$
and the induced action $\rho_0:\pi_1((B')^\circ, \ast)\to \operatorname{Aut}\left(H^d(X_{\ast},\mathbb Z)\right)$ via  the diagram:
\begin{equation}\label{mondiag}
\xymatrix
{ 
1\ar@{->}[r]&\pi_1((B')^\circ, \ast') \ar@{->}[r]^{f_{\#}}\ar@{->}[rd]^{\rho_0}&\pi_1(B^\circ ,\ast)\ar@{->}[r]^{\phi}\ar@{->}[d]^{\rho}&W(\Gamma)\ar@{->}[r] &1.\\
& &\operatorname{Aut}\left(H^d(X_{\ast},\mathbb Z) \right).&&\\
}
\end{equation}
More specifically, we are primarily interested in the unipotency of the monodromy associated to a generic simple loop  $\sigma\in \pi_1((B')^\circ, \ast')$ around the origin in $B'$; i.e. for a loop homotopic to the loop in Proposition \ref{promain}.  Recall, the monodromy associated to  $\sigma$ is always quasi-unipotent (e.g. \cite[p.6]{sgavii1}, \cite{landman}). The question is  closely related to the monodromy associated to  a generic simple loop $\gamma\in \pi_1(B^\circ, \ast)$ around the origin in $B_X$. It is well known that $\gamma=t_1\cdots t_n = \Pi\in \Art(\Gamma_X) \cong\pi_1(B^\circ,\ast)$ (up to conjugacy), i.e. the Artin-Coxeter element in the notation of the previous section.

To understand the monodromy action along $\gamma$ (and $\sigma$), we first consider the induced monodromy action   on the vanishing cohomology. Specifically, let $x\in X$ be a singular point of type $T_n\in\{A_n, D_n, E_n\}$. Restricting to a sufficiently small neighborhood $U$ of the singularity at $p$ in $\mathscr X$, it is well known that $X^v_{\ast}:=X_{\ast}\cap U$ is a Milnor fiber of the singularity and  $H^d(X^v_{\ast}, \mathbb Z)\cong \mathbb Z^n$; this is the so-called {\it vanishing cohomology of the singularity} (e.g. \cite[\S 2.1]{arnoldsum}). Furthermore, $H^d(X^v_{\ast}, \mathbb Z)$ is naturally endowed with a bilinear form (the intersection form), which is symmetric for $d$ even and alternating for $d$ odd. For any $d \equiv 2 \pmod 4$, $H^d(X^v_{\ast}, \mathbb Z)$ with this bilinear form is isometric to the standard root lattice of type $T_n$ (e.g. Arnold et al. \cite[Theorem p.129]{arnoldsum}).  Displacement of cycles along loops gives the  monodromy action on the vanishing cohomology: 
$$
\rho^v:\pi_1(B^\circ, \ast)\to \operatorname{Aut}\left(H^d(X^v_{\ast},\mathbb Z)\right),
$$
which is compatible with $\rho$.  A diagram similar to \eqref{mondiag} induces a morphism $\rho_0^v$ compatible with $\rho_0$.

The monodromy action $\rho^v$ along  $\gamma$, called the \textbf{classical monodromy operator}, is well understood in singularity theory \cite{arnoldsum}. Specifically, for $d \equiv 2 \pmod 4$, the monodromy action $\rho^v$ on $H^d(X^v_{\ast},\mathbb Z)\cong T_n$ factors through $W(\Gamma)$ in the modified version of \eqref{mondiag} to give the standard action of the Weyl group on the root lattice $T_n$. It follows that, in this situation, $\rho^v(\gamma)$ acts as the Coxeter element $s_1\cdots s_n\in W$. The general case follows from this via stabilization. Namely, it is known that there exists an isomorphism 
$\nu:H^2(X^{v,2}_{\ast},\mathbb Z)\to H^d(X^v_{\ast},\mathbb Z)$ such that
\begin{equation}\label{eqnsteq}
\rho^v(\gamma)=(-Id)^d\circ \nu\circ \rho^{v,2}(\gamma)\circ \nu^{-1},
\end{equation}
where $X^{v,2}_{\ast}$ is the Milnor fiber of the stably equivalent singularity of dimension $2$ and $\rho^{v,2}$ is its monodromy operator (e.g. 
\cite[\S 1.7, Theorem p.76, Theorem (iii) p.62]{arnoldsum}). With this, we conclude:

\begin{teo}\label{teomonvan} Let $B_X$ be a mini-versal deformation space of an ADE singularity $X$, let $B_X'\to B_X$ be the Weyl cover, and let $\sigma$ be a generic simple loop around the origin in $B_X'$.  Then
$$
\rho^v_0(\sigma)=\left\{ 
\begin{array}{rl} 
-Id & \text{if } \dim(X) \text{ is odd, and } T_n=A_n \text{ with } n \text{ even},\\
Id & \text{otherwise.}
\end{array}\right.
$$
\end{teo}
\begin{proof}
It follows from Proposition \ref{promain} that $\rho_0^v(\sigma)=\rho^v(\gamma^h)$.  From (\ref{eqnsteq}), we get
$$
\rho_0^v(\sigma)=(-Id)^{h\cdot\dim (X)},
$$
since $\rho^{v,2}(\gamma)^h=Id$ by the definition of the Coxeter number.  Finally, the Coxeter number is even, except in the case that $T_n=A_n$ for $n$ even. The result follows.  
\end{proof}

\begin{rem}
In order to better understand the exceptional case of the above theorem, one can also consider the so-called {\it spectrum of the singularity} (e.g. \cite[\S4.6]{arnoldsum}).  The stabilization  shifts the spectrum by $\frac{1}{2}$. It follows that in the odd dimensional $A_{2n}$ case, the eigenvalues of the classical monodromy operator are of order $2h$ rather than $h$. 
\end{rem}

We can now derive the following consequence for the monodromy representation.

\begin{cor}\label{teomono}
Let $X$ be a projective ADE variety of dimension $d$ satisfying \eqref{eqnmaincond} 
and let $B_X$ be a mini-versal deformation space of $X$.  Let $\tilde B_X\to B_X$ be the canonical log alteration with boundary $\tilde \Delta_X\subset \tilde B_X$.  Let $\sigma$ be a generic simple loop around an irreducible component of the boundary $D\subseteq \tilde \Delta_X$ of type $T_n$.   
Then the  monodromy action  $\rho_0(\sigma)$ is unipotent, except in the  case that $d$ is odd and  
$T_n=A_n$ with $n$ even. In that situation, $\rho_0(\sigma)$ is not unipotent, but its square is.
\end{cor}
\begin{proof}
Recall from Proposition \ref{proresdisc} that  the variety $X_b$ parameterized by a generic point $b\in D$ is an ADE variety with a unique singularity of type $T_n$.  In a sufficiently small  analytic neighborhood  of the image of $b$ in $B_X$, the mini-versal deformation $\mathscr X\to B_X$ is a versal deformation of $X_b$. Thus, without loss of generality, we can assume $X$ has a unique singularity of type $T_n$, and $\sigma$ is a generic loop around the main exceptional divisor in the blow-up $\tilde B_X\to B_X'$.  Letting $S\to B_X$ be a smooth morphism of the unit disc $S\subseteq \mathbb C$ sending $0\to b$, we may take $\sigma$ to be the image of a standard loop in $S^\circ$.

The family over $S$ is smooth outside of the fiber $X_0$  and $X_0$ has a unique singular point  $x\in X_0$, which is of type ADE.  We may assume further that the base point $\ast$ is chosen sufficiently close to $0\in S$ so that the fiber $X_\ast$ contains a Milnor fiber $X_\ast^v$ for the singularity.  In particular, note that  $H^d(X_\ast^v,\mathbb Z)$ is the vanishing cohomology of the singularity at $x\in X_0$. In this situation we obtain a long exact sequence of monodromy representations 
(e.g. \cite[pp.V, 7--9]{sgavii1}, \cite[(1.4)]{steenbrink})
\begin{equation}
\label{eqnmonoana}
\ldots \to H^d(X_{0},\mathbb C) \to H^d(X_{\ast},\mathbb C) \to H^d(X_\ast^v,\mathbb C)  \to H^{d+1}(X_{0},\mathbb C)\to \ldots
\end{equation}
and the action of $\pi_1(S^\circ,\ast)$ on  $H^i(X_{0},\mathbb C)$ is trivial (Grothendieck \cite[pp. V,4]{sgavii1}).
With some linear algebra, it follows easily that the monodromy action of $\sigma$ on  $H^d(X_{\ast},\mathbb C)$ is unipotent if and only if the monodromy action of $\sigma$ on the vanishing cohomology $H^d(X_\ast^v,\mathbb C)$  is unipotent.  The result is then a consequence of Theorem \ref{teomonvan}.
\end{proof}

%%%%%%%%%%%%%%%%%%%%%%%%%%%%%%%%%%%%%%%%%%%%%%%%%%%%%%%%%%%%%%%%%%
%%% S6 -- Morphism to stack
\section{Morphisms to the moduli stack  $\overline{\mathcal M}_g$}\label{secmstack}
In this section we discuss the problem of  lifting the morphism $\tilde B_X\to \overline M_g$ to the moduli stack. 
Our main tool is again the extension theorem of de Jong-Oort \cite[Theorem 5.1]{dejongoort}, which in its full strength (using \cite[Theorem 2.4]{dm}, \cite[VII 1  Cor. 3.8]{sgavii1}) implies:  \emph{Let $S$ be a smooth scheme and $\Delta$ an effective SNC divisor on $S$.  Set $S^\circ=S\setminus \Delta$, and suppose there is a  morphism $\phi^\circ:S^\circ \to \mathcal M_g$.  The monodromy of the associated  (analytic) family is unipotent along a (Zariski) open subset of each irreducible component of $\Delta$ if and only if there is a morphism $\phi:S\to \overline{\mathcal M}_g$ extending $\phi^\circ$}. 
Combining this with the results of the previous section, we get the following.

\begin{cor}\label{teossr1}
Let $X$ be an ADE curve with mini-versal deformation space $B_X$ and discriminant $\Delta_X$.  Let $\tilde B_X\to B_X$ be the canonical log alteration with discriminant $\tilde \Delta_X$.  Let $D$ be the union of the irreducible components of $\tilde \Delta_X$ corresponding to curves with $A_{2n}$ singularities.  Then there is a morphism $(\tilde B_X\setminus D)\to \overline{\mathcal M}_g$ extending the rational map $B_X\dashrightarrow  \overline{\mathcal M}_g$, but this morphism does not extend over $D$.
\end{cor}

\begin{proof}
This is a direct consequence of Corollary \ref{teomono} and the result of de Jong-Oort \cite[Theorem 5.1]{dejongoort}  mentioned above.  In fact, with the morphism $\tilde B_X\to \overline M_g$, the result can also be deduced from Corollary  \ref{teomono} using the Abramovich--Vistoli purity lemma \cite[Lemma 2.4.1]{av}.
\end{proof}

We now consider resolving completely the map $\tilde B_X\dashrightarrow \overline {\mathcal M}_g$ to obtain a lift of the morphism $\tilde B_X\to \overline M_g$.  
It is natural to consider the minimal modification of $\tilde B_X$ that admits a lift, namely:

\begin{dfn}[Canonical log alteration stack]\label{dfndc}
Let $X$ be an ADE curve 
and let $B_X$ be a mini-versal deformation space of $X$.  Let $\tilde B_X\to B_X$ be the canonical log alteration.
The \textbf{canonical log alteration stack} of the pair $(B_X,\Delta_X)$, denoted $\tilde{\mathcal  B}_X$, is defined as 
$$\tilde{\mathcal  B}_X:= \tilde B_X \times_{\overline M_g} \overline{\mathcal M}_g.$$
We set $\tilde \delta_X$  to be the pre-image of $\tilde \Delta_X$ in $\tilde {\mathcal B}_X$.
\end{dfn}

We are of course interested in describing the morphism $\tilde {\mathcal B}_X\to \tilde B_X$.  

\begin{cor}\label{teossr2}
The space $\tilde {\mathcal B}_X$ is a Deligne-Mumford stack with coarse moduli space $\tilde B_X$, admitting a morphism $\tilde {\mathcal B}_X\to \overline {\mathcal M}_g$ that resolves the rational map $\tilde B_X\dashrightarrow \overline{\mathcal M}_g$.  Moreover, $\tilde {\mathcal B}_X\to B_X$ admits a section outside of $D$, but any such section will fail to extend over $D$.
\end{cor}

\begin{proof}
This is essentially a restatement of the previous corollary.
\end{proof}

\begin{rem}\label{remhi}
Note that the stabilizer group of an object of $\tilde{\mathcal B}_X$ over a point in $\tilde B_X$ is isomorphic to the automorphism group of the associated stable curve.  In particular, at a generic point of an irreducible component of $D$, the stabilizer will include a subgroup isomorphic to $\mathbb Z/2\mathbb Z$.  Indeed from the results of Hassett \cite{hassettstable} (see Example \ref{exaansr}), in the $A_{2n}$ case, the automorphism group of the stable model will contain a $\mathbb Z/2\mathbb Z$ subgroup induced by the hyperelliptic involution of the tails.  
\end{rem}

\begin{rem}\label{remalt} Along these lines, an alternate proof of Corollary  \ref{teossr2}  can be given directly from Theorem \ref{teoresmm} and the results of Hassett \cite{hassettstable}, without the need for the complete monodromy arguments of \S\ref{secmono}. Assume for simplicity that the normalizations of the fibers of $\mathscr X\to B_X$ do not admit automorphisms.  The results of Hassett then imply that the automorphism groups of the stable models are trivial, except in the $A_{2n}$ case, as pointed out in the remark above. Thus $\tilde {\mathcal B}_X$ in fact agrees with $\tilde B_X$ outside of $D$.  It is then possible to make a direct argument that there is no section of $\tilde {\mathcal B}_X\to \tilde B_X$ that extends over $D$.
\end{rem}

 In order to resolve the moduli map while remaining in the category of schemes, one can use the monodromy computation to construct \'etale local branched covers.  For the sake of exposition, let us enumerate the components of $D$, setting $D=D_1+\ldots +D_N$, where the $D_i$ are irreducible.
At a point $b\in \tilde B_X$ where $\ell$ of the $D_i$ meet so that $D$ has a local equation of the form $x_1\ldots x_\ell$, there are two natural choices.  On the one hand, one may take the double cover given by
$t^2-x_1\cdots x_\ell$.  This introduces  a  toric singularity in the base, and  a resolution  can then be described log-\'etale locally using a theorem of  Mochizuki \cite[Theorem A(1)]{mochizuki}.   

An alternative approach is to take a higher degree covering; i.e. the $\mathbb Z_2^\ell$ cover  given by $t_1^2=x_1, \ldots, t_\ell^2=x_\ell$.
Since the base remains smooth, the de Jong--Oort theorem then gives an extension.  At the level of stacks, this \'etale local branch cover has the following description as a root stack.
 
\begin{dfn}[Canonical log alteration root stack]\label{dfnrs}
Let $X$ be an ADE curve 
and let $B_X$ be a mini-versal deformation space of $X$.  Let $\tilde B_X\to B_X$ be the canonical log alteration.
The \textbf{canonical log alteration root stack} of the pair $(B_X,\Delta_X)$, denoted $\tilde{\mathbb   B}_X$, is defined as 
$$
\tilde {\mathbb B}_X= \tilde B_X(\sqrt{D_1})\times_{\tilde B_X}\ldots \times_{\tilde B_X} \tilde B_X(\sqrt{D_N});
$$ 
i.e. the fibered product of the Cadman-Vistoli root stacks.
\end{dfn}

The following describes the relationship among $\tilde {\mathbb B}_X$, $\tilde {\mathcal B}_X$ and $\tilde B_X$.

\begin{cor}\label{corrs}
The space $\tilde {\mathbb B}_X$ is a smooth Deligne-Mumford stack with coarse moduli space $\tilde B_X$, admitting a morphism $\tilde {\mathbb B}_X\to \overline {\mathcal M}_g$, and hence to $\tilde {\mathcal B}_X$,  that resolves the rational map $\tilde B_X\dashrightarrow \overline{\mathcal M}_g$.  Moreover, $\tilde {\mathbb B}_X\to \tilde B_X$ is an isomorphism outside of $D$, and the stabilizer of an object of $\tilde {\mathbb B}_X$ lying over a generic point of an irreducible component of $D$ is isomorphic to $\mathbb Z/2\mathbb Z$.
\end{cor}

 \begin{proof}
The only thing to establish is the morphism $\tilde {\mathbb B}_X\to \overline {\mathcal M}_g$.  The discussion above gives an extension of the \'etale presentation of the stack.  One then uses the Abramovich-Vistoli purity lemma (see e.g. \cite[Proposition 3.4]{kl10}) to show that this descends to a morphism from the stack. 
\end{proof}

%%%%%%%%%%%%%%%%%%%%%%%%%%%%%%%%%%%%%%%%%
%%% S7 -- Explicit computation
\section{Explicit semi-stable reduction in the $A_n$ case}\label{secandirect}
In the $A_n$ case one can construct by hand the  simultaneous semi-stable reduction of Corollary \ref{teossr2}. The computation is essentially done in  Tjurina \cite[\S3]{tjurina0}, Brieskorn \cite{bruber} and in Fedorchuk  \cite[\S 5]{maksym} in slightly different co-ordinates.  We give a detailed account of the  computation here in the notation used in this paper.  We note that the computation can be done for higher dimensional singularities, and over an algebraically closed field $k=\bar k$ with $\operatorname{char}(k)=p>n+1$ (we need to avoid $p| m$ for any $m\le n+1$). We also discuss the monodromy obstruction to resolving the map to the moduli stacks in the $A_2$ case.

 \subsection{The Weyl cover}\label{secexpwc}
 The fundamental theorem of symmetric functions gives a natural $W(A_n)\cong\Sigma_{n+1}$ covering space of affine space:
 $$
 \operatorname{Spec} k[a_1,\ldots,a_{n+1}]=\mathbb A^{n+1}_k\stackrel{/\Sigma_{n+1}}\to \mathbb A^{n+1}_k= \operatorname{Spec} k[t_1,\ldots,t_{n+1}]
 $$
defined by 
$
t_i\mapsto \sigma_i(\underline a) \ \ i=1,\ldots,n+1
$, 
where $\sigma_i$ is the standard symmetric polynomial in $n+1$ variables, of degree $i$. 
 We can interpret this cover as mapping the set of all possible roots  mapping to the set of all monic polynomials of degree $n+1$ by
$$
(a_1,\ldots,a_{n+1})\mapsto \prod_{i=1}^{n+1} (x-a_i)=x^{n+1}+t_1x^{n}+\ldots+t_{n+1}.$$
The pull-back of the discriminant is then given by $
\prod_{i<j}(a_i-a_j)^2$.
By the fundamental theorem of symmetric functions, this can be described as a function of the basic symmetric polynomials:
$$
\prod_{i<j}(a_i-a_j)^2=D(\sigma_1(\underline a),\ldots,\sigma_{n+1}(\underline a)).
$$
The discriminant is thus given as $D(t_1,\ldots,t_{n+1})$, and pulls back via the Weyl cover to the product of hyperplanes.

\subsection{The mini-versal deformation space}\label{secexpdef}
Now let us consider the mini-versal deformation space of an $A_n$ singularity.  Since the singularity is given as
$$
X=\{x_1^{n+1}+x_2^2+\ldots+x_m^2=0\}\subseteq \mathbb A^m_k=\operatorname{Spec}k[x_1,\ldots,x_m],
$$
it is well known (e.g. \cite[Theorem 7.9, Corollary 7.10, Example 7.17]{vistolici}) that  the mini-versal deformation $\mathscr X\to B$ is given (on the level of rings) by
$$
k[t_2,\ldots,t_{n+1}]\to\frac{k[t_2,\ldots,t_{n+1}][x_1,\ldots,x_m]}{\left((x_2^2+\ldots+x_m^2)+x_1^{n+1}+t_{2}x_1^{n-1}+\ldots+t_{n+1}\right)}.
$$
Setting $f(x_1)=
 x_1^{n+1}+t_{2}x_1^{n-1}+\ldots+t_{n+1}$, we can see that the discriminant locus is exactly the set of $(t_2,\ldots,t_{n+1})$ where $f$ has a double root.  In other words, the base of the deformation is given by $\operatorname{Spec} k[t_1,\ldots,t_{n+1}]/(t_1)$, and the discriminant of the deformation is given by $\operatorname{Spec} k[t_1,\ldots,t_{n+1}]/(t_1,D(t_1,\ldots,t_{n+1}))$.

\subsection{The pull-back family on the Weyl cover $B'\to B$}\label{secexpweylfam}
There is an induced Weyl cover of the mini-versal deformation space.  
 $$
 \operatorname{Spec} k[a_1,\ldots,a_{n+1}]/(a_1+\ldots +a_{n+1})=B' \stackrel{/\Sigma_{n+1}}\to B= \operatorname{Spec} k[t_1,\ldots,t_{n+1}]/(t_1)
 $$
The pull-back of the discriminant is then given by
$$
k[a_1,\ldots,a_{n+1}]/(a_1+\ldots+a_{n+1},\prod_{1\le i<j\le {n+1}} (a_i-a_j)^2);
$$
i.e. the pull-back of the discriminant is a  union of hyperplanes on $B'$.  We can then consider $B'\times_B \mathscr X \to B'$, the pull-back of the family $\mathscr X\to B$.  On the level of rings, the family is given by
$$
k[a_1,\ldots,a_{n+1}]/(\sigma_1(\underline a))\to \frac{k[a_1,\ldots,a_{n+1}][x_1,\ldots,x_m]}{\left(\sigma_1(\underline a),(x_2^2+\ldots+x_m^2)+\prod_{i=1}^{n+1}(x_1-a_i)\right)}.
$$

\subsection{The wonderful blow-up of the discriminant: $\tilde B\to B'$}
\label{secexpblowupbase}
We will focus on the blow-up of the highest codimension stratum of the discriminant.  The rest follows inductively.  Consider the blow-up at the origin of $\operatorname{Spec}k[a_1,\ldots,a_{n+1}]$.  We will work with the standard coordinate patch 
$$\operatorname{Spec}k[b_1,\ldots,b_{n+1}]\to \operatorname{Spec}k[a_1,\ldots,a_{n+1}]$$
given by
$(b_1,\ldots,b_{n+1})\mapsto (b_1,b_1b_2,\ldots,b_1b_{n+1})$; the modifications for the other patches follow by symmetry.  Let $\tilde B\to B'$ be the blow-up at the origin, and let $\tilde B\times _{B} \mathscr X\to \tilde B$ be the pull-back of 
$B'\times_B\mathscr X\to B'$.
 On the level of rings, the family is given by
$$
k[b_1,\ldots,b_{n+1}]/(\sigma_1(1,b_2,\ldots,b_{n+1}))\to$$
\begin{equation}\label{eqnfambar}
 \frac{k[b_1,\ldots,b_{n+1}][x_1,\ldots,x_m]}{\left(\sigma_1(1,b_2,\ldots,b_{n+1}),(x_2^2+\ldots+x_m^2)+(x_1-b_1)\prod_{i=2}^{n+1}(x_1-b_1b_i)\right)}.
\end{equation}
Note that the exceptional divisor in $\tilde B$ is given locally by $b_1=0$, and that over this locus, there is a family of $A_n$ singularities.

\subsection{Blowing-up the total space ($n$ odd): extending the family over the generic point of the exceptional divisor in $\tilde B$}\label{secexpblowuptotalodd}
The locus of $A_n$ singularities  in the total space of $\tilde B \times_B {\mathscr X}$ is given by the ideal $(b_1,x_1,\ldots,x_m)$.  We would like to resolve these singularities.  To do this, we perform a blow-up 
along the ideal
$$
I=((b_1,x_1)^{\frac{n+1}{2}},x_2,\ldots,x_m).
$$
To be precise, define $\tilde {\mathscr X}\to \tilde B$ to be the blow-up of $\tilde B\times_B {\mathscr X}$ along the ideal $I$.  In the case of curves, one can easily check that the fiber over the generic point of the exceptional divisor in $\tilde B$ is a nodal curve.  More generally, in any dimension, the fiber consists of two irreducible components: the desingularization of the original fiber, sitting inside of 
$Bl_{(x_1^{\frac{n+1}{2}},x_2,\ldots,x_m)}\mathbb A^m_k$ (i.e. the weighted blow-up with weights $(\frac{n+1}{2},1\dots,1)$), and a smooth subvariety called a tail, sitting inside of $\mathbb P(1,1,\frac{n+1}{2},\ldots,\frac{n+1}{2})$.  
In the case of curves, the tail is a smooth hyperelliptic curve of genus $\lfloor n/2 \rfloor$ in $\mathbb P(1,1,\frac{n+1}{2})$.   Note these tails and their ambient space agree with the concrete examples in Hassett \cite[\S 6.2]{hassettstable}.

\subsection{Blowing-up the total space ($n$ even): extending the family over the generic point of the exceptional divisor in a $\mathbb Z/2\mathbb Z$ cover of $\tilde B$}\label{secexpblowuptotaleven}
Again, the locus of $A_n$ singularities  in the total space of $\tilde B \times_B {\mathscr X}$ is given by the ideal $(b_1,x_1,\ldots,x_m)$.  We would like to resolve these singularities.  
By virtue of the monodromy computations, we know that in order to do this we must first take $\tilde B'\to \tilde B$ a $2:1$ cover of the base branched along the exceptional divisor $\{b_1=0\}$.  On the level of rings, we have the branched cover of $\tilde B$ given by
$$
k[b_1,\ldots,b_{n+1}]\to k[c_1,\ldots,c_{n+1}]
$$
via $b_1\mapsto c_1^2$, and $b_i\mapsto c_i$ for $2=1,\ldots,n+1$. The total family after the base change to $\tilde B'$ is given by

$$
k[c_1,\ldots,c_{n+1}]/(\sigma_1(1,c_2,\ldots,c_{n+1}))\to$$
\begin{equation}\label{eqnfambareven}
 \frac{k[c_1,\ldots,c_{n+1}][x_1,\ldots,x_m]}{\left(\sigma_1(1,c_2,\ldots,c_{n+1}),(x_2^2+\ldots+x_m^2)+(x_1-c_1^2)\prod_{i=2}^{n+1}(x_1-c_1^2c_i)\right)}.
\end{equation}
Note that the divisor in $\tilde B'$ given locally by $\{c_1=0\}$ corresponds to the exceptional divisor in $\tilde B$, and that over this locus in $\tilde B'$, there is a family of $A_n$ singularities.

To obtain the semi-stable reduction, we perform a blow-up along the ideal
$$
I=\left((c_1^2,x_1)^{n+1}, (c_1^{n+1},c_1^{n-1}x_1,\ldots,c_1x_1^{n/2})\cdot(x_2,\ldots,x_m), \{x_ix_j\}_{2\le i,j\le m}\right).
$$
To be precise, define $\tilde {\mathscr X}'\to \tilde B'$ to be the blow-up of $\tilde B'\times_B {\mathscr X}$ along the ideal $I$.  In the case of curves, it is easy to check that the fiber over the generic point of the exceptional divisor in $\tilde B'$ is a nodal curve.  More generally, in any dimension, the fiber consists of two irreducible components: the desingularization of the original fiber, sitting inside of $Bl_{(x_1^{n+1},\{x_ix_j\}_{2\le i,j\le m})}\mathbb A^m_k$ and a smooth subvariety called a tail, sitting inside of $\mathbb P(1,2,n+1,\ldots,n+1)$.  In the case of curves, the tail is a smooth hyperelliptic curve of genus $n/2$ in $\mathbb P(1,2,n+1)$.  Note these tails and their ambient space agree with the concrete examples in Hassett \cite[\S 6.2]{hassettstable}.

\subsection{The monodromy obstruction in the $A_2$ case}\label{secnegex}
To better understand the fact that there exists an extension $\tilde B_X\to \overline{M}_g$ at the level schemes (Thm. \ref{teoresmm}), but not at the level of stacks in the  $A_{2n}$ case (Cor. \ref{teossr1}),  we discuss the $A_2$ case explicitly. 
  Thus, let $X$ be a cuspidal elliptic curve, and $\mathscr X\to B_X$ a mini-versal deformation.  Consider the family $\mathscr X'\to B_X'$ obtained via the Weyl cover, and the restriction  $(\sX')_{\mid L}\to L$ of this family to a generic arc $L$ through the origin in $B_X'$.

To show that there is no extension $\tilde B_X\to \overline{\mathcal M}_g$ to the moduli stack, it suffices to show that this restriction $(\sX')_{\mid L}\to L$ does not extend to a stable family of curves.  To see this, we observe that in the notation of the blow-up in \S \ref{secexpblowupbase}, the restriction $(\sX')_{\mid L}\to L$ is a surface $Z_{b_2}$ with equation (locally near the $A_2$ singularity):
\begin{equation}\label{fam3}
x_2^2+x_1^3-(b_2^2+b_2+1) b_1^2 x_1-b_2(1+b_2) b_1^3 =0,
\end{equation}
fibered in curves over $L$; here we view $b_1$ as a parameter for $L$ and $b_2$ as a fixed slope.
    
  The surface $Z_{b_2}$ has a $D_4$ singularity at the origin, which is also a cusp singularity for $X_0=X$, the central fiber of $Z_{b_2}$ viewed as a family of curves.  Recall the standard resolution of a $D_4$ surface singularity $x^2=f_3(y,z)$ is given by $4$ blow-ups. Namely,  after first blowing-up the $D_4$ singularity, giving the exceptional divisor $E_0$, the $D_4$ singularity ``splits'' into three $A_1$ singularities corresponding to the three roots of $f_3$. Then a simple blow-up  of each $A_1$ singularity  introduces the exceptional divisors $E_1,E_2,E_3$, giving the desired resolution.  We associate to this a  $\widetilde{D}_4$ graph (consisting of $E_0$ the central vertex, to which one attaches edges connecting the $4$ vertices corresponding to the curves $X_0$, $E_1$, $E_2$ and $E_3$).
 
At this point we can identify the monodromy obstruction via the theory of elliptic fibrations. Namely, the equation (\ref{fam3}) is a local equation near the $A_2$ singularity of the (generic) one-parameter deformation of a cusp. Thus, it can be used to compute the monodromy on the vanishing cohomology. On the other hand, (\ref{fam3}) can be viewed as giving a one parameter degeneration of elliptic curves. From the $\widetilde D_4$ graph, we conclude that this is a type $I_0^*$ degeneration in Kodaira's classification \cite[\S V.7, p.201]{barth}. It follows that the monodromy on the vanishing cohomology is minus the identity  (see \cite[p.210]{barth}).
 One should compare this also with the discussion of the elliptic involution from \cite[Ch. 2A]{hm}.

On the other hand, the extension to the moduli scheme is  without problems. Simply note that on $E_0$ there are $4$ distinct special points, corresponding to the intersections with 
$X_0$, $E_1,\dots E_3$. Thus, the one-parameter family $(\sX')_{\mid L}\to L$ determines a well defined tail $T$ and curve $[\tilde X\cup T]\in \overline{M}_g$ (see \S\ref{sechassettstable}). In fact, with respect to an appropriate choice of coordinates on $E_0\cong \bP^1$, the four special points on $E_0$ can be taken to be $\infty,1, b_2, -(1+b_2)$ respectively. It follows that the exceptional divisor of $\tilde B_X\to B_X'$ over the $A_2$ locus can be identified with the $\lambda$-line. Thus, at the level of tails, the moduli map $\tilde B_X\to \overline{M}_g$ can be described as mapping the $\lambda$-line to the $j$-line (see \cite[\S2.A]{hm} for a discussion). However, even though there exists a family of tails over the $\lambda$-line
$y^2=x(x-1)(x-\lambda)$,
they do not fit together to give a family of curves over  $\tilde B_X$  because of the elliptic involution ($y\to -y$)  acting on the tails. 

%%%%%%%%%%%%%%%%%%%%%%%%%%%%%%%%%%%%%%%%%
\bibliography{resolution}
\end{document}